\documentclass[10 pt]{amsart}

\newtheorem{theorem}{Theorem}[section]
\newtheorem{lemma}[theorem]{Lemma}

\theoremstyle{definition}

\theoremstyle{remark}
\newtheorem{remark}[theorem]{Remark}

\numberwithin{equation}{section}

\newcommand{\loc}{\operatorname{loc}}

\newcommand{\ecomp}{C_{c}^{\infty}(E)}

\newcommand{\lloc}{L_{\loc}}

\newcommand{\<}{\langle}
\renewcommand{\>}{\rangle}

\newcommand{\Dom}{\operatorname{Dom}}
\newcommand{\End}{\operatorname{End}}

\newcommand{\RE}{\operatorname{Re}}

\newcommand\RR{\mathbb{R}}

\newcommand{\mcomp}{C_{c}^{\infty}(M)}

\newcommand{\Del}{\Delta}

\newcommand{\n}{\nabla}

\newcommand{\sign}{\operatorname{sign}}

\begin{document}

\title[Schr\"odinger operators on Riemannian manifolds]{On a positivity preservation property for Schr\"odinger operators on Riemannian manifolds}

\author{Ognjen Milatovic}
\address{Department of Mathematics and Statistics\\
           University of North Florida   \\
           Jacksonville, FL 32224 \\
           USA
            }
\email{omilatov@unf.edu}

\subjclass[2010]{Primary 47B25,58J50; Secondary 35P05, 60H30}

\keywords{non-negative Ricci curvature, positivity preservation, Riemannian manifold, Schr\"odinger operator,
self-adjoint, singular potential}

\begin{abstract}
We study a positivity preservation property for Schr\"odinger operators with singular potential on geodesically complete Riemannian manifolds with non-negative Ricci curvature. We apply this property to the question of self-adjointness of the maximal realization of the corresponding operator.
\end{abstract}

\maketitle

\section{Introduction}\label{S:intro}
In his landmark paper~\cite{Kato72}, Kato proved a powerful distributional inequality, today known as Kato's inequality, which has since found numerous applications in self-adjointness (and $m$-accretivity) problems in $L^2(\mathbb{R}^n)$ for Schr\"odinger operators with a singular potential. For example, consider the operator $-\Delta +q$ with $q^{+}\in\lloc^2(\mathbb{R}^n)$ and $q^{-}\in L^{\infty}(\RR^n)+L^{n/2}(\RR^n)$, where $n\geq 5$, $q^{+}:=\max(q,0)$ and $q^{-}:=\max(-q,0)$.  Under these conditions, the operator $-\Delta +q$ is semi-bounded from below on $C_{c}^{\infty}(\mathbb{R}^n)$; see Lemma 2.1 in~\cite{Brezis-Kato79}. By an abstract fact (see Theorem X.26 in~\cite{rs}), to prove the essential self-adjointness of $-\Delta +q$ on $C_{c}^{\infty}(\mathbb{R}^n)$, it is enough to show that for any $v\in L^2(\mathbb{R}^n)$ such that $(-\Delta +q+\lambda)v=0$ in distributional sense, where $\lambda$ is a sufficiently large constant, we have $v=0$. To this end, we apply Kato's inequality to $v\in L^2(\mathbb{R}^n)$ satisfying $(-\Delta +q+\lambda)v=0$. This leads to the distributional inequality
\[
-\Delta|v|+\lambda|v|-q^{-}|v|\leq 0.
\]
The equality $v=0$ will follow if $q^{-}$ satisfies the positivity preservation property described below.

\medskip

\noindent\textbf{Positivity Preservation Property (PPP).} Let $F\in\lloc^{1}(\mathbb{R}^n)$ be a non-negative function. Then, there exists $\lambda_0\geq 0$ so that that if $\lambda>\lambda_0$, $u\in L^2(\mathbb{R}^n)$, $Fu\in \lloc^{1}(\mathbb{R}^n)$, and
\[
-\Delta u+\lambda u-Fu\geq 0,\qquad \textrm{in distributional sense},
\]
then $u\geq 0$.

\medskip
Br\'ezis and Kato showed in ~\cite{Brezis-Kato79} that (PPP) holds for (non-negative) functions
$F\in L^{\infty}(\RR^n)+L^{p}(\RR^n)$ with $p=\frac{n}{2}$ for $n\geq 3$, $p>1$ for $n=2$, and $p=1$ for $n=1$, together with the assumption $F\in \lloc^{n/2+\epsilon}(\RR^n)$, $\epsilon>0$,  in dimensions $n=3$ and $n=4$. The proof of (PPP) in~\cite{Brezis-Kato79} was based on elliptic equation theory and Sobolev space techniques.

Subsequently, using stochastic analysis techniques, Devinatz~\cite{D-80} showed that (PPP) holds for (non-negative) functions $F\in\lloc^{1}(\mathbb{R}^n)$ satisfying the property
\begin{equation}\label{E:devin}
\lim_{\alpha\to \infty}\left(\sup_{x\in\mathbb{R}^n}\frac{1}{4\pi^{n/2}}\int_{\mathbb{R}^n}\frac{F(x-y)}{|y|^{n-2}}
\left(\int_{\alpha y^2}^{\infty}\tau^{n/2-2}e^{-\tau}\,d\tau\right)\,dy\right)\,<\,1.
\end{equation}
We should note that the results of~\cite{D-80} include those of Jensen~\cite{J-78}. As an application of (PPP), the papers~
\cite{Brezis-Kato79, D-80, J-78} studied the self-adjointness problem of the corresponding Schr\"odinger operator.

In the context of a Riemannian manifold $M$, a simpler variant of (PPP) with $F\equiv 0$, which we label as (PPP-0), was considered in Proposition B.3 of~\cite{bms}, where it was shown that (PPP-0) holds under $C^{\infty}$-bounded geometry
assumption on $M$, that is,  $M$
has a positive injectivity radius and all Levi--Civita derivatives of the
curvature tensor of $M$ are bounded. The main point here is that the corresponding proof of~\cite{bms} depends on the existence of a sequence of smooth compactly supported functions $\chi_{k}$ with the following properties:

\begin{enumerate}

\item [(C1)] $0\leq \chi_{k}(x)\leq 1$, $x\in M$, $k=1,2,\dots$;

\item [(C2)] for every compact set $K\subset M$, there exists $k_{0}$ such
that $\chi_{k}=1$ on $K$, for $k\geq k_{0}$;

\item [(C3)]  $\sup_{x\in M}|d\chi_{k}(x)|\to 0$ as
$k\to\infty$.

\item [(C4)]  $\sup_{x\in M}|\Delta\chi_{k}(x)|\to 0$ as
$k\to\infty$.

\end{enumerate}

While the existence of a sequence $\chi_{k}$ satisfying (C1), (C2), and (C3) on an arbitrary geodesically complete Riemannian manifold is well known (see~\cite{ka}), a sequence satisfying all four properties has not yet been constructed (to our knowledge) in such a general context.

Very recently, G\"uneysu~\cite{Guneysu-2014} has improved (PPP-0) result considerably.  In particular, in the context of a geodesically complete Riemannian manifold with non-negative Ricci curvature, the author of~\cite{Guneysu-2014} has constructed a sequence $\chi_k$ satisfying (C1)--(C4) and proved (PPP-0). We should also note that the paper~\cite{Guneysu-2014} contains, among other things, a study of (PPP-0) in the setting of $L^{p}$ spaces with $p\in[1,\infty]$.

Let us point out that under $C^{\infty}$-bounded geometry
assumptions on $M$, an earlier study~\cite{Milatovic-12} showed that (PPP) holds for (non-negative) functions $F$ belonging to the Kato class (see Section~\ref{SS:Dyn-Kat} below) and satisfying the following additional assumption: $F\in L_{\loc}^p(M)$ with $p=n/2+\epsilon$, with some arbitrarily small $\epsilon>0$, for the case $2\leq n\leq 4$; $p=n/2$ for the case $n\geq 5$. We note that the paper~\cite{ Milatovic-12} used the latter assumption for elliptic equation and Sobolev space arguments. Based on recent developments in path-integral representations for semi-groups of Schr\"odinger operators with singular potential on Riemannian manifolds and the construction of cut-off functions satisfying (C1)--(C4) above, as seen in G\"uneysu's works~\cite{guneysu-2010, Guneysu-2012, Guneysu-2011, Guneysu-2014}, we will study (PPP) for a class functions $F$ that shares some properties with~(\ref{E:devin}) and includes, in particular, Kato class. In this regard, within the class of non-negative Ricci curvature, our results include those in~\cite{Milatovic-12}. In particular, we eliminate the assumption $F\in L_{\loc}^p(M)$ with $p$ as described above. Finally, as an application of the corresponding (PPP), we give sufficient conditions for the self-adjointness of the ``maximal" realization of the Schr\"odinger operator with electric potential whose negative part satisfies the same assumptions as $F$ in (PPP).

For reviews of results concerning the question of self-sdjointness of Schr\"odinger operators in $L^2(\RR^n)$ and $L^2(M)$, see, for instance,~\cite{ckfs} and~\cite{bms}.  For more recent studies, see the papers~\cite{bandara, bgp, GK, Guneysu-2014, GP}.

Finally, we remark that it might be possible to obtain a variant of (PPP) for perturbations of Dirichlet forms by measures. For the background on Dirichlet forms and their perturbations by measures, see, for instance, the book~\cite{Fuk}, papers~\cite{Kuwae-06, Kuwae-07, Stv96}, and references therein.

\section{Results}
\subsection{Notations}\label{SS:s-2-1} Let $M$ be a connected smooth Riemannian $n$-manifold without boundary.
Throughout the paper, by $\Delta$ we denote the corresponding \emph{negative} Laplace--Beltrami operator on $M$,  by $d\mu$ the volume measure of $M$, by $C^{\infty}(M)$ the space of complex-valued smooth functions on $M$, by $\mcomp$ the space of complex-valued smooth
compactly supported functions on $M$, by $\Omega^{1}(M)$ the space of
smooth 1-forms on $M$, by $L^2(M)$ the space of square integrable complex-valued functions on $M$, and by $(\cdot,\cdot)$ the usual inner product on $L^2(M)$. Additionally, $p(t,x,y)$ denotes the heat kernel of $M$ as in Theorem 7.13 in~\cite{Grigoryan-11}. We should emphasize that in this paper $p(t,x,y)$ corresponds to $e^{-t(-\Delta/2)}$, $t\geq 0$, instead of $e^{-t(-\Delta)}$.


\subsection{Positivity Preservation Property}
We are ready to formulate sufficient conditions for the positivity preservation property introduced in Section~\ref{S:intro}.

\begin{theorem}\label{T:main-1} Assume that $M$ is a geodesically complete connected Riemannian manifold with non-negative Ricci curvature. Let $F\colon M \to [0,\infty)$ be a measurable function satisfying the following property: there exists $t_0>0$ such that
\begin{equation}\label{E:assumption-star}
\sup_{x\in M}\left(\int_{0}^{t_0}\int_{M}p(s,x,y)F(y)\,d\mu(y)\,ds\right)<1.
\end{equation}
Then, there exists $\lambda_{*}\geq 0$ such that if $\lambda>\lambda_{*}$ and $u\in L^2(M)$ and $Fu\in\lloc^1(M)$ and $u$ satisfies the distributional inequality
\begin{equation}\label{E:brezis-kato-distributional}
(-\Del/2 -F+\lambda)u\geq 0,
\end{equation}
then $u\geq 0$ a.e.~on $M$.
\end{theorem}
\begin{remark} If $F$ belongs to Kato class, then~(\ref{E:assumption-star}) is satisfied; see Section~\ref{SS:Dyn-Kat} below.
\end{remark}

\subsection{Hermitian Vector Bundles and Bochner Laplacian} We will formulate our self-adjointness result for Schr\"odinger operators acting on Hermitian vector bundles over $M$. Before doing so, we explain some additional notations. Let $E\to M$ be a smooth Hermitian vector bundle over $M$ with underlying Hermitian structure $\langle\cdot, \cdot\rangle_{x}$ and the corresponding norms $|\cdot|_{x}$ on fibers $E_{x}$. Smooth sections of $E$ will be denoted by $C^{\infty}(E)$ and compactly supported smooth sections by $\ecomp$. With $d\mu$ as in Section~\ref{SS:s-2-1}, for all $1\leq p<\infty$ we obtain the $L_p$-spaces of sections $L^p(E)$ with norms $\|\cdot\|_{p}$. The space of essentially bounded sections of $E$ will be denoted by $L^{\infty}(E)$ with the corresponding norm $\|\cdot\|_{\infty}$. The notation $(\cdot,\cdot)_{L^2(E)}$ or just $(\cdot,\cdot)$, when there is no danger of confusion, stands for the usual inner product in $L^2(E)$.

Let $\nabla$ be a Hermitian connection on $E$ and let ${\nabla}^*$ be its formal adjoint with respect to $(\cdot,\cdot)_{L^2(E)}$.
In what follows, we will consider the so-called Bochner Laplacian operator ${\nabla}^*\nabla\colon C^{\infty}(E)\to C^{\infty}(E)$. For example, if we take $\nabla=d$, where $d\colon C^{\infty}(M)\to \Omega^{1}(M)$ is the standard differential, then $d^*d\colon C^{\infty}(M)\to C^{\infty}(M)$
is just the (non-negative) Laplace--Beltrami operator $-\Del$.

We are interested in the Schr\"odinger-type differential expression
\begin{equation}\label{E:defn-L-V}
L_V=\nabla^*\nabla/2+V,
\end{equation}
where $V$ is a measurable section of $\End E$ such that $V(x)\colon E_{x}\to E_{x}$
is a self-adjoint operator for almost every $x\in M$.

For every $x\in M$ we have the following canonical decomposition:
\begin{equation}\label{E:canonical}
V(x)=V^+(x)-V^-(x),
\end{equation}
where
\[
V^+(x):=P_+(x)V(x) \quad\textrm{ and }\quad  V^-(x):=-P_-(x)V(x),
\]
Here, $P_+(x):=\chi_{[0,+\infty)}(V(x))$ and $P_-(x):=\chi_{(-\infty,0)}(V(x))$, and $\chi_{G}$ denotes
the characteristic function of the set $G$.

Let $|V^{-}|$ denote the norm of the operator $V^{-}(x)\colon E_{x}\to E_{x}$, where $x\in M$. Thus, $|V^{-}|$ is a (real-valued) measurable function on $M$.


\subsection{Self-Adjoint Realization of $L_V$}\label{SS:maximal-operator} Assume that $V\in\lloc^1(\End E)$ and\\ $|V^{-}|\in\lloc^1(M)$.
We define $S$ as an operator in $L^2(E)$ by $Su=L_V u$
with the following domain $\Dom(S)$:
\begin{equation}\label{E:domhv2}
\{u\in L^2(E): V^{+}u\in\lloc^1(E),\,\,|V^{-}|u\in \lloc^{1}(E),\textrm{ and }L_{V}u\in L^2(E)\},
\end{equation}
Here, the expression $L_{V}u$ is understood in distributional sense.
\begin{theorem}\label{T:main-2} Assume that $M$ is a geodesically complete connected Riemannian manifold with non-negative Ricci curvature. Assume that $V^{+}\in\lloc^1(\End E)$ and $|V^{-}|$ satisfies the property~(\ref{E:assumption-star}).  Then $S$ is a self-adjoint operator.
\end{theorem}
\begin{remark} The assumption~(\ref{E:assumption-star}) on $|V^{-}|$ enables us to define the operator $L_{V}$ as a form sum; see Lemma~\ref{L:dom-vec} below.
\end{remark}
\begin{remark}
If $|V^{-}|$ satisfies~(\ref{E:assumption-star}), then $|V^{-}|\in\lloc^1(M)$; see Lemma~\ref{L:l-1-lloc-1} below.
\end{remark}
\begin{remark}\label{R:remark-subset} The condition $|V^{-}|u\in \lloc^{1}(E)$ in~(\ref{E:domhv2}), which we need in order to apply Theorem~\ref{T:main-1}, is stronger than the condition $V^{-}u\in\lloc^1(E)$. For the operator $L_V=-\Delta/2+V$ acting on scalar functions, the requirement $V^{+}u\in\lloc^1(E)$ and $|V^{-}|u\in \lloc^{1}(E)$ is equivalent to $Vu\in\lloc^1(M)$. In this case,~(\ref{E:domhv2}) describes the ``maximal" realization of $L_V$ in the sense of~\cite{Kato74}.
\end{remark}

\section{Proof of Theorem~\ref{T:main-1}}
We first recall two definitions from~\cite{Kuwae-07}.

\subsection{Contractive Dynkin Class and Kato Class.}\label{SS:Dyn-Kat} Let $p(t,x,y)$ be as in Section~\ref{SS:s-2-1}. We say that a measurable function $f\colon M\to\mathbb{R}$ belongs to \emph{contractive Dynkin class} relative to $p(t,x,y)$ and write $f\in S^{0}_{CD}$ if $|f|$ satisfies~(\ref{E:assumption-star}). We say that a measurable function $f\colon M\to\mathbb{R}$  belongs to \emph{Kato class} relative to $p(t,x,y)$ and write $f\in S^{0}_{K}$ if
\begin{equation}\label{E:assumption-kato-class}
\lim_{t\to 0+}\,\sup_{x\in M}\int_{0}^{t}\int_{M}p(s,x,y)|f(y)|\,d\mu(y)\,ds=0.
\end{equation}
Clearly, we have the inclusion $S^{0}_{K}\subset S^{0}_{CD}$.

\begin{remark} The term \emph{contractive Dynkin class} was suggested to the author of this paper by B.~G\"uneysu. We remark that the authors of~\cite{Kuwae-06} and \cite{Kuwae-07} use the term \emph{Dynkin class} for the class of measurable functions $f\colon M\to\mathbb{R}$ such that $|f|$ satisfies~(\ref{E:assumption-star}) with $1$ on the right hand side replaced by $\infty$.
\end{remark}

For $\alpha>0$, set
\[
r_{\alpha}(x,y):=\int_{0}^{\infty}e^{-\alpha t}p(t,x,y)\,dt.
\]
For $\alpha>0$ and $f\in S^{0}_{CD}$ define
\begin{equation}\label{E:c-alph}
c_{\alpha}(f):=\sup_{x\in M}\int_{M}r_{\alpha}(x,y)|f(y)|\,d\mu(y).
\end{equation}
By Lemma 3.2 in~\cite{Kuwae-07}, we have $c_{\alpha}(f)<\infty$ for all $\alpha>0$. We now set
\[
c(f):=\inf_{\alpha>0}c_{\alpha}(f).
\]
\begin{lemma}\label{L:l-1} If $f\in S^{0}_{CD}$ then $c(f)<1$.
\end{lemma}
\begin{proof}
By Lemma 3.1 in~\cite{Kuwae-06} (or Proposition 2.7(a) in~\cite{Guneysu-2011}), for any measurable function $f\colon M\to \mathbb{R}$ and all $\alpha,\, t>0$ we have
\begin{align}
&(1-e^{-\alpha t})\sup_{x\in M}\int_{M}r_{\alpha}(x,y)|f(y)|\,d\mu(y)\nonumber\\
&\leq \sup_{x\in M}\int_{0}^{t}\int_{M}p(s,x,y)|f(y)|\,d\mu(y)\,ds\nonumber.
\end{align}
Since $f\in S^{0}_{CD}$, there exists $t=t_0>0$ such that the right hand side is less than $1$. Consequently, we get
\[
c_{\alpha}(f)<\frac{1}{1-e^{-\alpha t_0}},
\]
for all $\alpha>0$, and from here $c(f)<1$ follows easily.
\end{proof}
The following lemma follows from Proposition 2.7(b) in~\cite{Guneysu-2011}:
\begin{lemma}\label{L:l-1-lloc-1} If $f\in S^{0}_{CD}$ then $f\in\lloc^1(M)$.
\end{lemma}

\begin{remark} The proof of Proposition 2.7(b) in~\cite{Guneysu-2011} uses strict positivity of $p(t,x,y)$, which requires connectedness of $M$.
\end{remark}

\begin{remark}\label{R:brownian-m} In the sequel, for any $x\in M$, the symbol $\mathbb{P}^{x}$ stands for the law of a Brownian motion $X_{t}$ on $M$ starting at $x$, and  $\mathbb{E}^{x}$ denotes the expected value corresponding to $\mathbb{P}^{x}$. Our hypothesis on $M$ ensure that $M$ is stochastically complete (see~\cite{guneysu-2010}); hence, the lifetime of $X_t$ is $\zeta=\infty$. We should emphasize that in this paper $\mathbb{P}^{x}$ is  $-\Delta/2$ diffusion, as opposed to $-\Delta$ diffusion.
\end{remark}
\begin{remark} We should note that the geodesic completeness and non-negative Ricci curvature assumptions are not used until Lemma~\ref{L:c-f-1} below. Also, in the absence of stochastic completeness, path-integral formulas below can be rewritten by taking into account the lifetime $\zeta$  of $X_t$.
\end{remark}

\begin{lemma}\label{L:l-2}
If $0\leq f\in S^{0}_{CD}$ then there exist constants $\beta>0$ and $\gamma>0$ such that
\begin{equation}\label{E:exp-1}
\sup_{x\in M}\,\mathbb{E}^{x}\left[e^{\int_{0}^{t}f(X_s)\,ds}\right]\leq \beta e^{\gamma t},
\end{equation}
for all $t>0$.
\end{lemma}
\begin{proof} First note that we can write
\[
\int_{0}^{t}\int_{M}p(s,x,y)f(y)\,d\mu(y)\,ds=\mathbb{E}^{x}\left[\int_{0}^{t}f(X_s)\,ds\right].
\]
By the definition of the class $S^{0}_{CD}$, there exists $t^{*}>0$ such that
\[
\nu_{t}:=\sup_{x\in M}\,\mathbb{E}^{x}\left[\int_{0}^{t}f(X_s)\,ds\right]<1,
\]
for all $0<t\leq t^*$.
By Khasminskii's Lemma (see Lemma 3.37 in~\cite{HL-2011}) we have
\[
\sup_{x\in M}\,\mathbb{E}^{x}\left[e^{\int_{0}^{t}f(X_s)\,ds}\right]\leq \frac{1}{1-\nu_{t}},
\]
for all $0<t\leq t^*$.
From here on we may repeat the proof of Lemma 3.38 of~\cite{HL-2011} to conclude that
\[
\sup_{x\in M}\,\mathbb{E}^{x}\left[e^{\int_{0}^{t}f(X_s)\,ds}\right]\leq\left(\frac{1}{1-\nu_{t^*}}\right)^{\lfloor t/t^*\rfloor +1},
\]
for all $t>0$, where $\lfloor a\rfloor:=\max\{k\in\mathbb{Z}\colon k\leq a\}$.

Setting $\beta=\frac{1}{1-\nu_{t^*}}$ and $\gamma=\frac{1}{t^*}\log\left(\frac{1}{1-\nu_{t^*}}\right)$, we obtain~(\ref{E:exp-1}).
\end{proof}

\subsection{Quadratic Forms}\label{SS:quad-forms}
In what follows, all quadratic forms are considered in the space $L^2(M)$.  Let $w\in\lloc^1(M)$. Set $w^{+}:=\max(w,0)$
and $w^{-}:=\max(-w,0)$, so that $w=w^{+}-w^{-}$. Define
\[
Q_0(u):=\frac{1}{2}\int_{M}|du|^2\,d\mu,
\]
with the domain $\textrm{D}(Q_0)=\{u\in L^2(M): Q_0(u)<\infty\}$. The form $Q_0$ is non-negative, densely defined (since $\mcomp\subset\textrm{D}(Q_0)$), and closed.
Define $Q_{w^{\pm}}(u):=\pm(w^{\pm}u,u)$ with the domain $\textrm{D}(Q_{w^{\pm}})=\big\{u\in L^2(M): w^{\pm}|u|^2\in L^1(M)\big\}$. The forms $Q_{w^{\pm}}$ are symmetric and densely defined (since $\mcomp\subset\textrm{D}(Q_{w^{\pm}})$). Note that the form $Q_{w^{+}}$ is non-negative.


\begin{lemma}\label{L:d-1} Assume that $w^{-}\in S^{0}_{CD}$. Then there exist $a\in [0,1)$ and $b\geq 0$ such that
\begin{equation}\label{E:d-1}
|Q_{w^{-}}(u)|\leq a|Q_0(u)|+b\|u\|^2,\quad\textrm{for all }u\in\textrm{D}(Q_0).
\end{equation}
\end{lemma}
\begin{proof}
Let $c_{\alpha}(w^{-})$ be as in~(\ref{E:c-alph}).
We have already observed that $w^{-}\in S^{0}_{CD}$ implies $c_{\alpha}(w^{-})<\infty$ for all $\alpha>0$. By Theorem 3.1 in~\cite{Stv96} we have
\[
(w^{-}u,u)\leq \frac{c_{\alpha}(w^{-})}{2}\int_{M}|du|^2\,d\mu+\alpha c_{\alpha}(w^{-})\|u\|^2,
\]
for all $u\in\textrm{D}(Q_0)$ and all $\alpha>0$. By Lemma~\ref{L:l-1} we have $c(w^{-})<1$. Hence, there exists $\alpha^{*}$ such that $c_{\alpha^{*}}(w^{-})<1$, which shows~(\ref{E:d-1}).
\end{proof}
By Theorem VI.1.11 in~\cite{Kato66} and Example VI.1.15 in \cite{Kato66}, the form $Q_{w^{+}}$ is closed. By Lemma~\ref{L:d-1} above and Theorem VI.1.33 in~\cite{Kato66}, the form $Q_{0,w}:=(Q_0+Q_{w^{+}})+Q_{w^{-}}$ is densely defined, closed and semi-bounded from below with $\textrm{D}(Q_{w})=\textrm{D}(Q_0)\cap \textrm{D}(Q_{w^{+}})\subset \textrm{D}(Q_{w^{-}})$. Let $H(w)$ denote the semi-bounded from below self-adjoint operator in $L^2(M)$ associated to $Q_{0,w}$ by Theorem VI.2.1 of~\cite{Kato66}.

\subsection{Semigroup Associated to $H(-w^{-})$.}\label{SS:sg} As seen from the proof of Lemma~\ref{L:d-1}, for $w^{-}\in S^{0}_{CD}$, there exists $\alpha_{*}$ such that $c_{\alpha^{*}}(w^{-})<1$, and the form $Q_{0,-w^{-}}:=Q_{0}+Q_{w^{-}}$ is semi-bounded from below by $-\alpha_{*}c_{\alpha_{*}}(w^-)$. Let $H(-w^{-})$ be the corresponding self-adjoint (semi-bounded from below) operator and let $U_{2,-w^{-}}(t):=e^{-tH(-w^{-})}$, $t\geq 0$, be the corresponding $C_0$-semigroup in $L^2(M)$. The following Lemma was proven in Theorem 3.3 of~\cite{Stv96}:

\begin{lemma}\label{L:d-2} Assume that $w^{-}\in S^{0}_{CD}$. Then, the operators $U_{2,-w^{-}}(t)$ act as $C_0$-semigroups in $L^p(M)$, for all $p\in[1,\infty)$, and we label those semigroups as $U_{p,-w^{-}}(t)$. Moreover, there exist $C\geq 0$ and $\omega\in\mathbb{R}$ (depending only on $\alpha_{*}$ and $c_{\alpha_{*}}(w^-)$) such that
\begin{equation}\label{E:n-s-1}
\|U_{p,-w^{-}}(t)\|_{L^p\to L^p}\leq Ce^{\omega t},
\end{equation}
for all $p\in[1,\infty)$ and $t\geq 0$.
\end{lemma}
\subsection{Path Integral Representation of $U_{2,-w^{-}}(t)$.} Let $X_t$ be as in Remark~\ref{R:brownian-m}. For $w^{-}\in S^{0}_{CD}$ we have the Feynman--Kac formula
\begin{equation}\label{E:feyn-kac-1}
(U_{2,-w^{-}}(t)g)(x)=\mathbb{E}^{x}\left[e^{\int_{0}^{t}w^{-}(X_s)\,ds}g(X_t)\right],
\end{equation}
for all $g\in L^2(M)$, all $t\geq 0$, and a.e. $x\in M$. In the Kato-class case $w^{-}\in S^{0}_{K}$, the formula~(\ref{E:feyn-kac-1}) was proven in Theorem 2.9 of~\cite{Guneysu-2012}. The same proof works for $w^{-}\in S^{0}_{CD}$ thanks to~(\ref{E:d-1}) and the the following property: $w^{-}\in S^{0}_{CD}$ implies
\[
\mathbb{P}^{x}[w^{-}(X_{\bullet})\in\lloc^1[0,\infty)]=1,\qquad \textrm{a.e.}~x\in M.
\]
For the latter property see the proof of~Lemma 2.4(b) in~\cite{Guneysu-2012}, which works without change for the class  $S^{0}_{CD}$ instead of $S^{0}_{K}$.
\begin{lemma}\label{L:l-3} If $w^{-}\in S^{0}_{CD}$ then for all $g\in L^2(M)\cap L^{\infty}(M)$ and all $t\geq 0$ we have
\[
\|U_{2,-w^{-}}(t)g\|_{\infty}\leq \beta e^{\gamma t}\|g\|_{\infty},
\]
where $\beta>0$ and $\gamma>0$ are some constants.
\end{lemma}
\begin{proof} The lemma follows by combining~(\ref{E:feyn-kac-1}) and~(\ref{E:exp-1}).
\end{proof}
\subsection{Cut-off Functions.} The following lemma was proven in Theorem 2.2 of~\cite{Guneysu-2014}.
\begin{lemma}\label{L:c-f-1} Assume that $M$ is a geodesically complete Riemannian manifold with non-negative Ricci curvature. Then there exists a sequence of
functions $\chi_{k}\in\mcomp$ satisfying the properties (C1)--(C4) from Section 1.
\end{lemma}
\subsection{Sobolev Space.} Let $\widetilde{H}^{2}(M)$ denote the space of measurable functions $u\colon M\to\mathbb{C}$ such that
\[
\|u\|_{\widetilde{H}^{2}}:=\|u\|+\|du\|+\|\Delta u\|<\infty,
\]
where $\|du\|$ denotes the norm in $L^2(\Lambda^1T^*M)$.
\begin{lemma}\label{L:s-n-k} Assume that $M$ is a geodesically complete Riemannian manifold with non-negative Ricci curvature. Let $0\leq u\in \widetilde{H}^{2}(M)$. Then there exits a sequence of functions $0\leq u_k\in\mcomp$ such that $\|u_k-u\|_{\widetilde{H}^{2}}\to 0$, as $k\to\infty$.
\end{lemma}
\begin{proof} In this proof, $(\widetilde{H}^{2}(M))^{+}$ and $(\mcomp)^{+}$ denote the sets of non-negative elements of $\widetilde{H}^{2}(M)$ and $\mcomp$ respectively. Let
$u\in(\widetilde{H}^{2}(M))^{+}$ and let $\chi_k$ be the sequence of cut-off functions as in Lemma~\ref{L:c-f-1}.
We will first show that the set of compactly supported elements of $(\widetilde{H}^{2}(M))^{+}$ is dense in $(\widetilde{H}^{2}(M))^{+}$. To do this, first note that
\begin{equation}\label{E:lb-1}
d(\chi_k u)=ud\chi_k+\chi_k du
\end{equation}
and
\begin{equation}\label{E:lb-2}
\Delta(\chi_ku)=\chi_k(\Delta u)+2\langle d\chi_k,du \rangle+u(\Delta\chi_k).
\end{equation}
If we denote the Riemannian metric of $M$ by $r=(r_{jk})$, the notation $\langle\kappa,\psi\rangle$ in~(\ref{E:lb-2}) for $1$-forms
$\kappa=\kappa_j dx^j$ and $\psi=\psi_k dx^k$
means
\[
\langle\kappa,\psi\rangle:=r^{jk}\kappa_j\psi_k,
\]
where $(r^{jk})$ is the inverse matrix to $(r_{jk})$, and the standard Einstein summation convention is understood.
Now the property $\|\chi_ku-u\|_{\widetilde{H}^{2}}\to 0$, as $k\to\infty$, easily follows from~(\ref{E:lb-1}),~(\ref{E:lb-2}), and (C1)--(C4). This shows that the set of compactly supported elements of $(\widetilde{H}^{2}(M))^{+}$ is dense in $(\widetilde{H}^{2}(M))^{+}$. It remains to show that $(\mcomp)^{+}$ is dense in the set of compactly supported elements of $(\widetilde{H}^{2}(M))^{+}$. To see this, we start with a compactly supported element $u\in (\widetilde{H}^{2}(M))^{+}$. Since the support of $u$ is compact, using a partition of unity, we may assume that $u$ is supported in a coordinate chart $(G,\phi)$ of $M$ such that $\phi(G)=K_1$, where $K_1$ is an open ball of radius $1$ in $\mathbb{R}^n$. Applying the Friedrichs mollification procedure to $u\circ\phi^{-1}$, we obtain a sequence of non-negative smooth functions $v_{j}$ with support in $K_1$ converging to $u\circ\phi^{-1}$ with respect to $\|\cdot\|_{W^{2,2}}$, as $j\to\infty$, where $\|\cdot\|_{W^{k,p}}$ stands for the usual Sobolev norm in $\mathbb{R}^{n}$, with $k$ indicating the highest derivative and $p$ the corresponding $L^p$-space. Then $v_j\circ\phi$ converges to $u$ in the norm $\|\cdot\|_{\widetilde{H}^{2}}$, as $j\to\infty$.
\end{proof}
With the above preparations, the proof of Theorem~\ref{T:main-1} proceeds as that of (PPP) in~\cite{D-80}.

\smallskip

\noindent\textbf{Proof of Theorem~\ref{T:main-1}.}  Let $F$ be as in hypotheses of the Theorem. Define $F_k:=\min(F,k)$, $k\in\mathbb{Z}_{+}$, and consider the semigroup $U_{2,-F_{k}}(t)$ as in Section~\ref{SS:sg}. Denote the generator of this semigroup by $H(-F_k)$. As $F_k\in L^{\infty}(M)$ and $M$ is geodesically complete, it is well known that $(-\Del/2-F_k)|_{\mcomp}$ is essentially self-adjoint and its (self-adjoint) closure is $\overline{(-\Del/2-F_k)|_{\mcomp}}u=(-\Del/2-F_k)u$, for all
\[
u\in\Dom(\overline{(-\Del/2-F_k)|_{\mcomp}})=\{u\in L^2(M)\colon \Del u\in L^2(M)\}.
\]
Furthermore, by Theorem VI.2.9 in~\cite{Kato66}, the operator $\overline{(-\Del/2-F_k)|_{\mcomp}}$ coincides with  $H(-F_{k})$, which, in turn, coincides with the operator sum $H(0)-F_k$, where $F_k$ stands for the corresponding multiplication operator by the function $F_k$.

Noting $-F_k\geq -F$ and using the representation~(\ref{E:feyn-kac-1}) together with~(\ref{E:n-s-1}) we have
\begin{equation}\label{E:e-k}
\|U_{2,-F_k}(t)\|_{L^2\to L^2}\leq \|U_{2,-F}(t)\|_{L^2\to L^2}\leq Ce^{\omega t},
\end{equation}
where $U_{2,-F}(t)$ is the semigroup corresponding to $H(-F)$ as in Section~\ref{SS:sg}.

Let $\lambda_{*}:=\max\{\omega, \gamma, \alpha_{*}c_{\alpha_{*}}(F)\}$, where $\gamma$ is as in Lemma~\ref{L:l-2} and $\alpha_{*}c_{\alpha_{*}}(F)$ is as in Section~\ref{SS:sg}. For $\lambda>\lambda_{*}$ the (linear) operator
$(\lambda+H(-F_k))^{-1}\colon L^2(M)\to L^2(M)$ is bounded. Let $g\in L^2(M)\cap L^{\infty}(M)$ and $g\geq 0$. For $k\in\mathbb{Z}_{+}$, define
\begin{equation}\label{E:u-k-dfn-1}
u_k:=(\lambda+H(-F_k))^{-1}g.
\end{equation}
Using the representation
\begin{equation}\label{E:dfn-u-k-1}
(\lambda+H(-F_k))^{-1}g=\int_{0}^{\infty}e^{-\lambda t}U_{2,-F_k}(t)g\,dt,
\end{equation}
the estimate~(\ref{E:e-k}) and the inequality $\lambda>\lambda_{*}$, we obtain
\begin{equation}\label{E:e-u-k}
\|u_k\|\leq C\int_{0}^{\infty}e^{-(\lambda-\omega)t}\|g\|\,dt\leq C_1\|g\|.
\end{equation}
for all $k\in\mathbb{Z_{+}}$, with some constant $C_1\geq 0$.

Note that $u_k\geq 0$ by~(\ref{E:dfn-u-k-1}),~(\ref{E:feyn-kac-1}) and the assumption $g\geq 0$.
By Lemma~\ref{L:l-3} we have
\begin{align}\label{E:bnd-inf}
0\leq u_k(x)&=\int_{0}^{\infty}e^{-\lambda t}U_{2,-F_k}(t)g\,dt\nonumber\\
&\leq \int_{0}^{\infty}e^{-\lambda t}\|U_{2,-F_k}(t)g\|_{\infty}\,dt\nonumber\\
&\leq \beta\int_{0}^{\infty}e^{-(\lambda-\gamma)t}\|g\|_{\infty}\,dt\leq C_2\|g\|_{\infty},
\end{align}
where $C_2\geq 0$ is a constant, and in the last inequality we used $\lambda>\lambda_{*}\geq \gamma$.

By the definition of $u_k$ we have
\[
(\lambda+H(-F_k))u_k=g.
\]
Taking the inner product in $L^2(M)$ with $u_k$, using the fact that $H(0)$ is the operator associated to the form $Q_{0}$, and recalling the inequality $-F_k\geq-F$, we obtain
\begin{align}
&(g,u_k)=((\lambda+H(0)-F_k)u_k,u_k)= \lambda\|u_k\|^2+\frac{1}{2}\int_{M}|du_k|^2\,d\mu-(F_ku_k,u_k)\nonumber\\
&\geq \lambda\|u_k\|^2+\frac{1}{2}\int_{M}|du_k|^2\,d\mu-(Fu_k,u_k),\nonumber
\end{align}
which, upon combining with~(\ref{E:d-1}) and rearranging, leads to
\[
(g,u_k)\geq \frac{1-a}{2}\int_{M}|du_k|^2\,d\mu+(\lambda-b)\|u_k\|^2.
\]
From the last inequality we get
\begin{align}\label{E:eqn-d-1}
&\frac{1-a}{2}\int_{M}|du_k|^2\,d\mu\leq (g,u_k)+(b-\lambda)\|u_k\|^2\nonumber\\
&\leq |b-\lambda|(C_1)^2\|g\|^2+C_1\|g\|^2,
\end{align}
where in the last estimate we used Cauchy--Schwarz inequality and~(\ref{E:e-u-k}).

Let $0\leq \psi\in\mcomp$, let $u$ be as in the hypothesis of the theorem, and let $0\leq g\in L^2(M)\cap L^{\infty}(M)$ and $u_k$ be as in
~(\ref{E:u-k-dfn-1}).  We have the following equality:
\begin{equation}\label{E:sobolev-k}
(\psi u, g)=(\psi u, (-\Delta/2+\lambda-F_k)u_k).
\end{equation}
Using~(\ref{E:eqn-d-1}) and the property
\[
0\leq u_k\in\Dom(H(-F_k))=\{v\in L^2(M)\colon \Delta v\in L^2(M)\},
\]
we have $0\leq u_k\in \widetilde{H}^2(M)$. Thus, by Lemma~\ref{L:s-n-k}, without loss of generality, we may assume that $0\leq u_k\in\mcomp$ in~(\ref{E:sobolev-k}), which we will do from now on.

Using~(\ref{E:lb-1}) and~(\ref{E:lb-2}) we have
\begin{align}
&(\psi u, (-\Delta/2+\lambda-F_k)u_k)=((-\Delta/2)(\psi u),u_k)+\lambda(u,\psi u_k)\nonumber\\
&+((F-F_k)\psi u,u_k)-(F\psi u, u_k)\nonumber\\
&=((-\Delta/2+\lambda-F)u,\psi u_k)+(((-\Delta/2) \psi)u, u_k)\nonumber\\
&-(du, (d\psi)u_k)+((F-F_k)\psi u, u_k)\nonumber\\
&\geq (((-\Delta/2) \psi)u, u_k)-(du, (d\psi)u_k)+((F-F_k)\psi u,u_k)\nonumber,
\end{align}
where in the last inequality we used $0\leq \psi u_k\in\mcomp$ and the assumption~(\ref{E:brezis-kato-distributional}).
Using the fact that $-\Delta =d^*d$ we have
\begin{align}
(du, (d\psi)u_k)&=(u,d^*((d\psi)u_k))=(u,(d^*d \psi)u_k)-(ud\psi,du_k)\nonumber\\
&=((-\Delta\psi)u,u_k)-(ud\psi,du_k),
\end{align}
which upon combining with the preceding estimate and~(\ref{E:sobolev-k}) leads to
\begin{equation}\label{E:c-d-l-1}
(\psi u, g)\geq ((F-F_k)\psi u,u_k)+((\Delta/2)\psi)u,u_k)+(ud\psi,du_k).
\end{equation}
Let us replace $0\leq\psi\in\mcomp$ by a sequence  $\chi_m$ of cut-off functions from Lemma~\ref{L:c-f-1}.
Using~(\ref{E:e-u-k}),~(\ref{E:eqn-d-1}), and the properties of $\chi_m$, it is easy to see that the last two terms on the right hand side of~(\ref{E:c-d-l-1})
converge to $0$ as $m\to\infty$. We now consider the term $((F-F_k)\chi_m u,u_k)$. For a fixed $m\in\mathbb{Z}_{+}$, using the property~(\ref{E:bnd-inf}) we have
\[
(F-F_k)\chi_m uu_k\to 0,\quad\textrm{a.e.}~x\in M,\qquad \textrm{as }k\to\infty.
\]
and
\[
|(F-F_k)\chi_m uu_k|\leq C_3\chi_{m}F|u|\in L^1(M),
\]
where $C_3\geq 0$ is some constant.
Thus, by dominated convergence theorem we have
\[
((F-F_k)\chi_{m}u,u_k)\to 0,\qquad \textrm{as }k\to\infty.
\]
Returning to~(\ref{E:c-d-l-1}) and using our convergence observations, together with property (C2) of $\chi_m$, we obtain $(u,g)\geq 0$. Since $0\leq g\in L^2(M)\cap L^{\infty}(M)$ is arbitrary, we get $u\geq 0$ a.e. $\hfill\square$

\section{Proof of Theorem~\ref{T:main-2}} Let $E$, $\nabla$, and $V$ be as in hypotheses of Theorem~\ref{T:main-2}.
We begin by describing $L^2(E)$ analogues of quadratic forms from Section~\ref{SS:quad-forms}.
\subsection{Quadratic Forms in Vector-Bundle Setting}
Define
\[
Q_{\nabla,0}(u):=\frac{1}{2}\int_{M}|\nabla u|^{2}\,d\mu
\]
with the domain $\textrm{D}(Q_{\nabla,0})=\{u\in L^2(E): \nabla u\in L^2(T^*M\otimes E)\}$. Note that $Q_{\nabla,0}$ is non-negative, densely defined, and closed.
Next we define $Q_{V^{\pm}}(u)=\pm(V^{\pm}u,u)$ with the domain $\textrm{D}(Q_{V^{\pm}})=\big\{u\in L^2(E):\langle
V^{\pm}u,u\rangle\in L^1(M)\big\}$. The forms $Q_{V^{\pm}}$ are densely defined and symmetric.  Note that the form $Q_{V^{+}}$ is non-negative.

\begin{lemma}\label{L:dom-vec} Let $V^{-}$ be as in hypotheses of Theorem~\ref{T:main-2}. Then there exist $a\in[0,1)$ and $b\geq 0$ such that
\begin{equation}\label{E:|V_2|u^2<<|du|^2+|u|^2-modified}
\int_{M}\langle V^{-}u,u \rangle\,d\mu\leq (a/2)\|\nabla u\|_{L^2(T^*M\otimes E)}^2+b\|u\|_{L^2(E)}^2,
\end{equation}
for all $u\in \textrm{D}(Q_{\nabla,0})$.
\end{lemma}
\begin{proof}
Let $u\in \textrm{D}(Q_{\nabla,0})$ and let $Q_0$ be as in Section~\ref{SS:quad-forms}. By Corollary 2.5 in~\cite{Guneysu-2011} we have $|u|\in \textrm{D}(Q_{0})$,  and
\begin{equation}\label{E:guneysu}
\|d|u|\|_{L^2(\Lambda^1T^*M)}^2\leq \|\nabla u\|_{L^2(T^*M\otimes E)}^2.
\end{equation}
Using~(\ref{E:d-1}) and~(\ref{E:guneysu}) we obtain
\begin{align}
&\int_{M}\langle V^{-}u,u \rangle\,d\mu\leq \int_{M}|V^{-}||u|^2\,d\mu \leq (a/2)\|d|u|\|_{L^2(\Lambda^1T^*M)}^2+b\||u|\|_{L^2(M)}^2\nonumber\\
&\leq (a/2)\|\nabla u\|_{L^2(T^*M\otimes E)}^2+b\|u\|_{L^2(E)}^2,\quad\textrm{for all }u\in\textrm{D}(Q_{\nabla,0}),\nonumber
\end{align}
where $a$ and $b$ are as in~(\ref{E:d-1}).
\end{proof}
By Theorem VI.1.11 in~\cite{Kato66} and Example VI.1.15 in~\cite{Kato66}, the form $Q_{V^{+}}$ is closed. As a consequence of Lemma~\ref{L:dom-vec}, analogously as in Section~\ref{SS:quad-forms}, the form $Q_{\nabla,V}:=Q_{0}+Q_{V^{+}}+Q_{V^{-}}$ is densely defined, closed and semi-bounded from below with $\textrm{D}(Q_{\nabla,V})=\textrm{D}(Q_{\nabla,0})\cap \textrm{D}(Q_{V^{+}})\subset \textrm{D}(Q_{V^{-}})$. Let $H_{\nabla}(V)$ denote the semi-bounded from below self-adjoint operator in $L^2(E)$ associated to $Q_{\nabla,V}$.

\subsection{Description of $H_{\nabla}(V)$}\label{SS:decription-op}
By Lemma~\ref{L:l-1-lloc-1} we have $|V^{-}|\in\lloc^1(M)$, which together with~(\ref{E:|V_2|u^2<<|du|^2+|u|^2-modified}) and geodesic completeness of $M$, means that the hypothesis of Theorem 1.2 in~\cite{Milatovic-12} are satisfied. The latter theorem  gives the following description of $H_{\nabla}(V)$:
\[
\Dom(H_{\nabla}(V))=\{u\in\textrm{D}(Q_{\nabla,0}) \colon \langle V^{+}u,u\rangle \in L^1(M)\textrm{ and }L_Vu\in L^2(E)\}
\]
and $H_{\nabla}(V)u=L_{V}u$, for all $u\in\Dom(H_{\nabla}(V))$, where $L_{V}$ is as in~(\ref{E:defn-L-V}).

In the proof of Theorem~\ref{T:main-2} we will use Kato's inequality for
Bochner Laplacian, whose proof is given in Theorem 5.7 of~\cite{bms}.
\begin{lemma}\label{L:kato} Let $M$ be a connected Riemannian
manifold (not necessarily geodesically complete). Let $E$ be a Hermitian vector bundle over $M$,
and let $\nabla$ be a Hermitian connection on $E$. Assume that
$w\in\lloc^1(E)$ and $\n^*\n w\in\lloc^1(E)$. Then
\begin{equation}\label{E:kato}
      -\Delta |w| \ \leq \ \RE\< \n^*\n w,\sign w\>_{E_x},
\end{equation}
where
\[
      \sign w(x) \ = \
      \left\{\begin{array}{cc}
          \frac{w(x)}{|w(x)|} & \textrm{if\ \ $w(x)\neq0$ },\\
          0 &\textrm{ otherwise}.
       \end{array}\right.
\]
\end{lemma}
\begin{remark} The original version of Kato's inequality was proven in~\cite{Kato72}.
\end{remark}
\noindent\textbf{Proof of Theorem~\ref{T:main-2}}
Note that for all $u\in\Dom(H_{\nabla}(V))$ we have $\langle V^{+}u,u\rangle \in L^1(M)$ and $\langle V^{-}u,u\rangle \in L^1(M)$, where the latter inclusion follows by~(\ref{E:|V_2|u^2<<|du|^2+|u|^2-modified}). Thus, as observed in~(4.3) of~\cite{Milatovic-12}, the mentioned two inclusions and hypotheses on $V$ imply $V^{+}u\in\lloc^1(E)$ and $|V^{-}|u\in\lloc^1(E)$. Now we just compare the descriptions of $H_{\nabla}(V)$ and $S$ to conclude that the operator relation $H_{\nabla}(V)\subset S$ holds.

It remains to prove that $\Dom(S)\subset\Dom(H_{\nabla}(V))$.  Let
$u\in\Dom(S)$.  Let $\lambda_{*}$ be as in Theorem~\ref{T:main-1}. Since $H_{\nabla}(V)$ is a semi-bounded from below (self-adjoint) operator, we can select $\lambda>\lambda_{*}$ large enough so that $H_{\nabla}(V)+\lambda$ is a positive self-adjoint operator. With this selection of $\lambda$, the operator $(H_{\nabla}(V)+\lambda)^{-1}\colon L^2(E)\to L^2(E)$ is bounded. Since $u\in\Dom(S)$, we may define
\[
v:=(H_{\nabla}(V)+\lambda)^{-1}(S+\lambda)u
\]
and write
\[
(H_{\nabla}(V)+\lambda)v=(S+\lambda)u.
\]
Since $H_{\nabla}(V)\subset S$, we can rewrite the last equality as
\begin{equation}\label{E:kato-w-regular}
(S+\lambda)w=0,
\end{equation}
where $w:=u-v$.

Since $w\in \Dom(S)$, we have $V^{+}w\in\lloc^1(E)$ and $|V^{-}|w\in\lloc^1(E)$. Furthermore, from~(\ref{E:kato-w-regular}) we get
$$(\nabla^*\nabla/2)w=-Vw-\lambda w\in\lloc^1(E).$$
By Lemma~\ref{L:kato} we have
\begin{align}
&-(\Delta/2)|w|\leq\RE\< (\n^*\n/2) w,\sign w\>_{E_x}=\RE\langle-(V+\lambda)w,\sign w\rangle_{E_x}\nonumber\\
&\leq (|V^{-}|-\lambda)|w|,\nonumber
\end{align}
which leads to
\begin{equation}\label{E:kato-w-new-1}
(-\Del/2-|V^{-}|+\lambda)|w|\leq 0.
\end{equation}

Since $|V^{-}||w|\in\lloc^1(M)$, we may use Theorem~\ref{T:main-1} with $F=|V^{-}|$ to conclude $|w|\leq 0$
a.e.~on $M$. This shows that $w=0$ a.e.~on $M$, i.e.~$u=v$ a.e.~on
$M$; therefore, $u\in\Dom(H_{\nabla}(V))$. $\hfill\square$

\section*{Acknowledgement} We are grateful to Batu G\"uneysu for numerous fruitful discussions.
We take the opportunity to express our thanks to the anonymous referee for valuable suggestions and helping us improve the presentation of the material.

\bibliographystyle{amsplain}

\end{document}